\documentclass[12pt]{article}
\usepackage[toc,page]{appendix}
\usepackage{color}
\usepackage{amsfonts}
\usepackage{amsmath}
\usepackage{amssymb}
\usepackage{graphicx}
\usepackage{mathptmx}
\usepackage{diagbox}
\usepackage{subfigure}
\usepackage{float}
\usepackage{parskip}
\RequirePackage[colorlinks, hyperindex, plainpages=false]{hyperref}
\usepackage{amsmath}
\usepackage{dsfont}
 \usepackage{sectsty}

\def\P{{\mathbb P}}
\def\E{{\mathbb E}}
\def\Z{{\mathbb Z}}
\def\N{{\mathbb N}}
\def\R{{\mathbb R}}

\def\1{{\mathbf 1}}

\numberwithin{equation}{section}

\newtheorem{theorem}{Theorem}[section]
\newtheorem{lemma}[theorem]{Lemma}
\newtheorem{corollary}[theorem]{Corollary}
\newtheorem{proposition}[theorem]{Proposition}

\newenvironment{proof}[1][Proof.]{\textbf{#1} }{\hfill $\blacksquare$}
\def\beq{\begin{equation}}
\def\eeq{\end{equation}}
\newcommand{\bei}{\begin{itemize}}
	\newcommand{\eei}{\end{itemize}}
\newcommand{\ben}{\begin{enumerate}}
	\newcommand{\een}{\end{enumerate}}
\newcommand{\beqn}{\begin{eqnarray}}
\newcommand{\beqnn}{\begin{eqnarray*}}
\newcommand{\eeqn}{\end{eqnarray}}
\newcommand{\eeqnn}{\end{eqnarray*}}
\newcommand{\brm}{\begin{rmk}}
\newcommand{\erm}{\end{rmk}}

\begin{document}
\title{Extinction time of an epidemic with infection age dependent infectivity}


\author{
Anicet Mougabe--Peurkor \footnote{  \scriptsize{Universit\'e F\'elix Houphouet Boigny, Abidjan, C\^{o}te d’Ivoire. mougabeanicet@yahoo.fr
}}, \quad 
Ibrahima~Dram\'{e}  \footnote{ \scriptsize{Universit\'e Cheikh Anta Diop de Dakar, FST, LMA, 16180 Dakar-Fann, S\'en\'egal. iboudrame87@gmail.com
}},
\quad 
Modeste N'zi \footnote{ \scriptsize{Universit\'e F\'elix Houphouet Boigny, Abidjan, C\^{o}te d’Ivoire. modestenzi@yahoo.fr}}
\\ \quad and \quad 
Etienne~Pardoux \setcounter{footnote}{6}\footnote{ \scriptsize {Aix Marseille Univ, CNRS, I2M, Marseille, France
. etienne.pardoux@univ-amu.fr}}
}

\maketitle


\begin{abstract}
This paper studies the distribution function of the time of extinction of a subcritical epidemic, when a large 
enough proportion of the population has been immunized and/or the infectivity of the infectious individuals has been reduced, so that the effective reproduction number is less than one. We do that for a SIR/SEIR model, where infectious individuals have an infection age dependent infectivity, as in the model introduced in the 1927 seminal paper of Kermack and McKendrick \cite{KMK}. Our main conclusion is that simplifying the model as an ODE SIR model, as it is largely done in the epidemics literature, introduces a biais toward shorter extinction time.
\end{abstract}

\vskip 3mm 

\section{Introduction}

Consider an epidemic which is declining: the number $M$ of infected individuals is moderate, and decreases,
while the total size $N$ of the population is much larger.
In such a phase, the approximation by the deterministic model is no longer valid. Rather, as the initial phase of an epidemic, the final phase can be well approximated by a branching process, in this case a subcritical branching process. The extinction time is thus random. It is of interest to have some information on the distribution function of this extinction time. Indeed, if the subcriticality is due in part to some rules imposed to the population, like mask wearing in public transport, classrooms, workplace, theaters etc., it is important to evaluate how long such rules must be maintained. 

Our epidemic model is a SIR/SEIR model, i.e. we assume that after having been infected and having recovered, an individual remains immune to the disease for ever. This is not quite realistic. However, if the duration of the studied period is not too long, then the number of individuals who loose their immunity during that period can be neglected.
On the other hand, the stochastic SIR/SEIR model upon which we base our analysis is non Markov. Following the ideas of Kermack--McKendrick \cite{KMK} and Forien, Pang, Pardoux \cite{RPE}, we consider a model where 
the infectivity of each infectious individual is infection age dependent (and random, the realizations corresponding to various individuals being i.i.d.). We characterize the distribution function of the extinction time 
of the approximating non Markov branching process with a single ancestor as the unique solution of a Volterra--type integral equation, for which we give a converging numerical approximation. The derivation of the equation is based upon a methodology introduced by Crump and Mode \cite{CM68}. From this result, we deduce in
Theorem \ref{th_et} a formula for the time we have to wait after $t_0$ for the epidemic to go extinct, if at time $t_0$ we have $M$ infected individuals in a population of size $N$, with $M<<N$.

With the help of a numerical scheme, we compute an approximation of the distribution function of the time of extinction, and compare the result with the distribution function of the extinction time of a Markov branching process which approximates the classical Markov SIR model (whose law of large numbers limit is the most standard SIR ODE model), which is known explicitly. This comparison is done between two models which have both the same effective reproduction number $R_{eff}$ (the mean number of ``descendants''  one infectious individual has at this stage of the epidemic), and the same rate $\rho$ of continuous time exponential decrease. Our conclusion is that the usual ODE SIR model leads to an underestimation of the extinction time.

Our work was inspired by the recent work of Griette et al. \cite{QGM20}, where the authors neglect the new infections during the final phase. Note that this approximation is justified by the data, in the case of the end of the Covid epidemic  in Wuhan in 2020.
Our work does not make such a simplifying assumptions, and allows a very general law for the varying infectivity, and a completely arbitrary law for the duration of the infectious period.

The paper is organized as follows. We present our varying infectivity SIR model in Section 2, together with its branching process approximation. In Section 3, we study the distribution function of the extinction time of the branching process.  In section 4, we present several examples of SIR/SEIR models, including the classical ODE SIR model, ODE SEIR model, and we specify the type of varying infectivity  which we have in mind. In section 5,
we compare the time of extinction of the branching approximations to our varying infectivity model, and to the 
ODE SIR model. In section 6, we discuss the results obtained in that comparison. Finally in section 7 (the Appendix), we establish the convergence of a numerical approximation scheme of the equation established in section 3.

{\bf Notations} In what follows we shall use the following notations: $\mathbb{Z} = \{..., -2, -1, 0, 1, 2, ...\}$, $\R=(-\infty,\infty)$, $\R_+=[0,\infty)$ and $\R_-=[0,\infty)$. For $x\in \R_+$, $[x]$ denotes the integer part of $x$ and $\lceil x \rceil$ ( resp. $ \lfloor x \rfloor$) denotes the ceiling function (resp. the floor function). For $x\in \R$, $x^+$ (resp. $x^-$) denotes the positive part of $x$ (resp. the negative part of $x$). For $(a,b)\in \R^2$, $a<b$, $\mathcal{U}([a,b])$ denotes the uniform distribution on $[a,b]$. $D([0,\infty))$ denotes the space of functions from $[0,\infty)$ into $\mathbb{R}$ which are right continuous and have left limits at any $t>0$. We shall always equip the space $D([0,\infty))$ with the Skorohod topology, for the definition of which we refer the reader to Billingsley \cite{Bill}.

\section{The SIR model with Varying Infectivity}
\subsection{The epidemic model}
Let $\{\lambda_j(t),\ t\ge0\}$,  $j\in\Z\backslash\{0\}$ be a collection of mutually independent non negative functions, which are such that the $\{\lambda_j\}_{j\ge 1}$ are identically distributed, as well as the $\{\lambda_j\}_{j\le -1}$. We assume that these function belongs $a.s.$ to $D([0,\infty))$.
We consider a SIR model which is such that the $j$--th initially infected individual has the infectivity $\lambda_{-j}(t)$ at time $t$, while the $j$--th individual infected after time $0$ has at time $t$ the infectivity $\lambda_j(t-\tau_j)$, if $0<\tau_1<\cdots<\tau_\ell<\cdots$ denote the successive times of infection in the population. The quantity $t-\tau_j$ is the age of infection of individual $j$ at time $t$. Note that we assume that $\lambda_j$ vanishes on $\R_{-}$. The newly infected individual is chosen uniformly at random in the population, and if that individual is susceptible, then it jumps from the S to the I compartment at its time of infection while nothing happens if the individual is not susceptibe.  Examples of function $\lambda_j(t)$ will be given below. That function can be first zero during the exposed period, then the individual becomes infectious, and at age of infection
$\eta_j=sup\{t,\ \lambda_j(t)>0\}$, the individual recovers (i.e. jumps into the R compartment) and is immune for ever. Clearly an important quantity is the total force of infection in the population at time $t$: $\mathfrak{F}^N(t)$, which is the sum of all the infectivities of the infected individuals at that time.
Here $N$ is the total number of individuals in the population. The sum of the numbers of individuals in the three compartments is constant in time : $S^N(t)+I^N(t)+R^N(t)=N$ for all $t\ge0$. For $X=S, \mathfrak{F}, I, \text{ or }R$, 
we define the renormalized quantity $\bar{X}^N(t)=X^N(t)/N$. The main result of \cite{RPE} is that as $N\to\infty$, 
$(\bar{S}^N(t),\bar{\mathfrak{F}}^N(t),\bar{I}^N(t), \bar{R}^N(t))\to (\bar{S}(t),\bar{\mathfrak{F}}(t),\bar{I}(t), \bar{R}(t))$, where the limit is the unique solution of the following system of integral equations, which already appears in the seminal paper of Kermack and McKendrick \cite{KMK}:
\begin{align}\label{AB}
\begin{cases}
\bar{S}(t)&=\bar{S}(0)-\int_0^t \bar{S}(s)\bar{\mathfrak{F}}(s)ds,\\
\bar{\mathfrak{F}}(t)&=\bar{I}(0)\bar{\lambda}^0(t)+\int_0^t\bar{\lambda}(t-s)\bar{S}(s)\bar{\mathfrak{F}}(s)ds,\\
\bar{I}(t)&= \bar{I}(0)F_0^c(t)+\int_0^t F^c(t-s)\bar{S}(s)\bar{\mathfrak{F}}(s)ds, \\
\bar{R}(t)&=\bar{R}(0)+\bar{I}(0)F_0(t)+\int_0^t F(t-s)\bar{S}(s)\bar{\mathfrak{F}}(s)ds\,,
\end{cases}
\end{align}
where $\bar{\lambda}^0(t)=\E[\lambda_{-1}(t)]$ and $\bar{\lambda}(t)=\E[\lambda_1(t)]$, $F_0$ (resp. $F$)
is the distribution function of $\eta_{-1}$ (resp. of $\eta_1$) and $F_0^c(t)=1-F_0(t)$, $F^c(t)=1-F(t)$. This convergence holds true provided that $\lambda\in D$ a.s. and for some $\lambda^\ast>0$, $0\le\lambda_j(t)\le\lambda^\ast$ a.s. for all $j\in\mathbb{Z}$ and $t\ge0$, see \cite{RGE3}. The original proof in \cite{RPE} puts more restrictions on $\lambda$.

\subsection{The branching process approximation}
Suppose that at time $t_0$ only a moderate number $M<<N$ of individuals in the population is infected, and the mean number $R_{eff}=\bar{S} (t_0)\int_0^\infty\bar{\lambda}(t)dt$ of individuals which an infected individual infects satisfies $R_{eff}<1$. Then the epidemic is declining. It can be well approximated by the following non Markovian continuous time branching process. We will study its extinction time in te next section, and deduce a good approximation of the time we have to wait after $t_0$ for the epidemic to go extinct. Note that we approximate the proportion $\bar{S}(t)$ by $\bar{S}(t_0)$, for any $t\ge t_0$.

We consider the branching process $Z(t)$ with (to start with) a single ancestor at time $0$. The $j$--th individual in the population, independently from all other individuals, lives for a duration $\eta_j$, and during its life time gives birth to children (one at a time) at rate $\bar{S} (t_0)\lambda_j(t)$, where $(\eta_j,\lambda_j)$ are as above. This is clearly a non Markov continuous time branching process, which belongs to the class of Crump--Mode--Jagers branching processes. 

\section{The extinction time of the branching process associated to the varying infectivity model}\label{sec3}
For now on, $t\ge0$ stands for $t-t_0$ ($t\ge t_0$).
We define $\hat{\lambda}(t):=\bar{S} (t_0)\lambda(t)$. Let $Z(t)$ denote the number of descendants at time $t$ of an individual born (i.e. infected) at time $0$, in the continuous time branching process which describes the number of infected individuals at time $t$. 
This ancestor infects susceptible individuals during the time interval $[0,\eta]$, at the random and varying rate $\hat{\lambda}(t)$. His descendants have the same behaviour, each one independently from all the others.

In this paper, we make the following assumption on the infectivity function. 

\textbf{Assumption} {(\bf H)}  We shall assume that there exists a constant $\lambda^* >0$ such that 
\[ \lambda(t)\le \lambda^* \ \mbox{almost surely, for all }t\ge0.\]
Let $T_{ext}  = \inf\{t>0 : Z(t) = 0\}$ denote the extinction time of the epidemic, $G(s,t) = \mathbb{E}\left(s^{Z_t} \right)$, $|s|\le1$, denote the probability generating function of $Z(t)$ and $F(t)=G(0,t)$ the distribution function of the extinction time. 

\subsection{Distribution function of the extinction time}
In this subsection, we will characterize the distribution function of the extinction time of $Z$ as the unique solution of an integral equation. To this end, we imitate the computations done in the proof of Theorem 4.1 in \cite{CM69}. We first start by determining the generating function $G(s,t)$ of $Z$ in order next to deduce the distribution function of the extinction time. 

Denote by $Z_0(t)$ the descendance of the ancestor at time $t$, and for $j\ge1$, $Z_j(t)$ the descendance of the $j$--th direct descendant of the ancestor at time $t$ after its birth. Then $\{Z_j(.) , j\ge  0\}$ is a sequence of independent and identically distributed (i.i.d) random processes which have the  law of $Z$. In order to simplify our notations, we will write $\hat{\lambda}_0$ (resp. $\eta_0$) for the value of $\hat{\lambda}$ (resp. $\eta$) associated with $Z_0$ . Formula (3.1) from \cite{CM68} reads
\begin{equation}\label{eq:31}
Z_0(t)=\mathds{1}_{\eta_0 >t}+\sum_{j=1}^{Q_0(t)} Z_j(t-t^{j}),
\end{equation}
where $Q_0(t)$ is the number of direct descendants of the ancestor born on the time interval $(0,t]$. Moreover, $Q_0(t)$ is a counting process, which conditionnally upon $\hat{\lambda}_0(\cdot)$, is a nonhomogeneous Poisson process with varying intensity $\hat{\lambda}_0(t)$, and $0<t^1<t^2<\cdots$ are the successive jump times of the process $Q_0(t)$. 

We have
\begin{proposition}\label{p1}
	The probability generating function $G$ satisfies the following integral equation
	\begin{align*}
	G(s,t) = \E\left[s^{\mathds{1}_{\eta>t}}\exp\Bigg\{\int_{0}^{t}\Big(G(s,t-u)-1\Big)\hat{\lambda}(u)du\Bigg\}\right].
	\end{align*}
\end{proposition}
\begin{proof} 
Since $Z$ has the same law as $Z_0$, we first to compute $\mathbb{E}\left[s^{Z_0(t)}|\hat{\lambda}_0 \right]$ in order to deduce the value of $G$. From \eqref{eq:31}, we deduce that 
\begin{align*}
\mathbb{E}\left[s^{Z_0(t)}|\hat{\lambda}_0 \right]=&
 \sum_{k=0}^{\infty} s^{\mathds{1}_{\eta_0 >t}} \mathbb{P}(Q_0(t) = k|\hat{\lambda}_0)\mathbb{E}\left\{\prod_{j=1}^{k}s^{Z_j(t-t^j)}\Big|Q_0(t) = k, \hat{\lambda}_0\right\} \\
=& \sum_{k=0}^{\infty}s^{\mathds{1}_{\eta_0>t}}\mathbb{P}(Q_0(t) = k|\hat{\lambda}_0) \mathbb{E}\left\{\prod_{j=1}^{k} G(s,t-t^j)\Big|Q_0(t) = k, \hat{\lambda}_0\right\}\\
=& \sum_{k=0}^{\infty}s^{\mathds{1}_{\eta_0>t}}\mathbb{P}(Q_0(t) = k|\hat{\lambda}_0) \frac{k!}{\Big(\int_{0}^{t}\hat{\lambda}_0(v)dv\Big)^k}\times \\
&\int_{0}^{t}\int_{0}^{u_k}...\int_{0}^{u_2}\prod_{j=1}^{k}G(s,t-u_j)\hat{\lambda}_0(u_1)...\hat{\lambda}_0(u_k)du_1...du_k\\
=& s^{\mathds{1}_{\eta_0>t}}\exp\Bigg(-\int_{0}^{t}\hat{\lambda}_0(v)dv\Bigg)\times \\
&\sum_{k=0}^{\infty} \int_{0}^{t}\int_{0}^{u_k}...\int_{0}^{u_2}\prod_{j=1}^{k}G(s,t-u_j)\hat{\lambda}_0(u_1)...\hat{\lambda}_0(u_k)du_1...du_k\\
=& s^{\mathds{1}_{\eta_0>t}}\exp\Bigg(-\int_{0}^{t}\hat{\lambda}_0(v)dv\Bigg)\sum_{k=0}^{\infty}  \frac{1}{k!}\Bigg(\int_{0}^tG(s,t-u)\hat{\lambda}_0(u)du\Bigg)^k
\\ 
=& s^{\mathds{1}_{\eta_0>t}}\exp\Bigg\{\int_{0}^{t}\Big(G(s,t-u)-1\Big)\hat{\lambda}_0(u)du\Bigg\}.
\end{align*}
The third equality exploits the well known result on the law of the times of the jumps of a Poisson process on a given interval, given the number of those jumps (see Exercise 6.5.4 in \cite{PE}, which treats the case of a constant rate, the general case follows via an obvious time change), and the fourth equality the conditional law of $Q_0(t)$, given $\hat{\lambda}_0$.  We thus obtain that
\begin{align*}
G(s,t) = \E\left[s^{\mathds{1}_{\eta_0>t}}\exp\Bigg\{\int_{0}^{t}\Big(G(s,t-u)-1\Big)\hat{\lambda}_0(u)du\Bigg\}\right].
\end{align*}
Since $(\hat{\lambda}_0,\eta_0)$ has the same law as $(\hat{\lambda},\eta)$, we can drop the subindices 0 in the last formula, yielding the formula of the statement.
\end{proof} 

The term $s^{\mathds{1}_{\eta>t}}$ can be written as follows:
$s^{\mathds{1}_{\eta>t}} = \mathds{1}_{\eta\leq t} + s\mathds{1}_{\eta>t}$. From this, we deduce readily the following Corollary for $F(t)=G(0,t)$.
\begin{corollary}\label{cor1}  The distribution function $F$ of the extinction time of the branching process with one unique ancestor born at time $0$ satisfies the following integral equation:
\begin{align}\label{ef1}
F(t) = \E\left[ \mathds{1}_{\eta \leq t}\exp\Bigg\{\int_{0}^{t}\Big(F(t-u)-1\Big)\hat{\lambda}(u) du\Bigg\}\right].
\end{align}
\end{corollary}
The fact that \eqref{ef1} characterizes $F$ follows from the following crucial result. 
\begin{proposition}\label{unicite}
	Equation \eqref{ef1} has a unique $[0,1]$-valued solution.
\end{proposition}
\begin{proof} The distribution function of the extinction time solves this equation. Let us show that this equation has at most one $[0,1]$-valued solution. To this end, suppose that the equation has two solutions $F^1$ and $F^2$ which are upper bounded by $1$. We have 
	\begin{equation*}
	F^1(t)-F^2(t) =  \E \Bigg[ \mathds{1}_{\eta \leq t}\Bigg(\exp\Bigg\{\int_{0}^{t}\Big(F^1(t-u)-1\Big)\hat{\lambda}(u) du\Bigg\}-\exp\Bigg\{\int_{0}^{t}\Big(F^2(t-u)-1\Big)\hat{\lambda}(u) du\Bigg\}\Bigg)\Bigg].
	\end{equation*}
 From the fact that $|e^{-x}-e^{-y}|\le |x-y|$, $\forall x,y>0$, we deduce that 
	\begin{align*}
	\Big|F^1(t)-F^2(t)\Big| &\leq  \E \left[\int_{0}^{t}\hat{\lambda}(u)\Big|F^1(t-u)-F^2(t-u)\Big| du\right]\\
	&\leq  \hat{\lambda}^*\int_{0}^{t}\Big|F^1(u)-F^2(u)\Big| du,
	\end{align*}
where we have used assumption {(\bf H)} and the notation $\hat{\lambda}^\ast=\bar{S}(t_0)\lambda^\ast$.	 The desired result follows by combining this with Gronwall's lemma.
\end{proof}

\subsection{Epidemic starting at time $\chi<0$.}
Now let us consider the case where the ancestor has been infected at a random time $\chi<0$. Then the total progeny at time $t$ of this ancestor can be written as follows:
\begin{equation*}
Z_0(t)=\mathds{1}_{\eta_0 >t+\chi}+\sum_{j=1}^{Q_0(t)} Z_j(t-t^{j}).
\end{equation*}
From an easy adaptation of the argument used in the proof of Proposition \ref{p1}, we deduce the  
\begin{proposition}\label{p1p} The distribution function $F_\chi$ of the extinction time of the epidemic  starting with a unique ancestor at time $0$, who was born at time $\chi(<0)$ satisfies :
\begin{align*}
F_\chi(t) = \E\left[ \mathds{1}_{\eta \leq t+\chi}\exp\left\{\int_{0}^{t}\Big(F_\chi(t-u)-1\Big)\hat{\lambda}(u-\chi) du\right\}\right].
\end{align*}
\end{proposition}

\subsection{Epidemic with multiple infected at the initial time.}
In the first two subsections, we have considered an epidemic that starts with a single infected individual. In this subsection, we consider an epidemic that starts with $M\in \N$ infected individuals at the initial time. The goal is to determine the distribution of the extinction time. To this end,  let $(\hat{\lambda}_i, \eta_i)_{1\le i\le M}$ be a sequence of pairs of random variables where $\hat{\lambda}_i$ (resp. $\eta_i$) denote the infectivity (resp. the lifetime) of the ancestor $i$. Let $(u_i)_{1\le i\le M}$ be a sequence of independent and identically distributed (i.i.d) random variables with law $\mathcal{U}([0,1])$. Note that the sequence $\left(\hat{\lambda}_i,\eta_i, u_i\right)_{1\le i \le M}$ is i.i.d  and for each $i$, we assume that $(\hat{\lambda}_i, \eta_i)$ and $u_i$ are independent. We assume that the individual $i$ was infected at time $\chi_i=-u_i\eta_i$, which we believe is the most natural model. 

From Proposition \ref{p1p}, we know that the distribution of the extinction time of the epidemic starting with the ancestor $i$, is given by
\begin{align}\label{ee}
\tilde{F}(t) = \E\left[ \mathds{1}_{\tilde{\eta} \leq t}\exp\left\{\int_{0}^{t}\Big(\tilde{F}(t-r)-1\Big)\tilde{\lambda}(r) dr\right\}\right],
\end{align}
with $\tilde{\eta}_i = \eta_i(1-u_i)$ and $\tilde{\lambda}_i(t) = \hat{\lambda}_i(t-\chi_i)$. Since, as the dynamics of reproduction remains the same for all infected individuals resulting from each ancestor, hence from the branching property, we deduce the main result of this section.
\begin{theorem}\label{th_et}
The distribution function of the time we have to wait in order to see the extinction of the epidemic, if at time $t_0$ we have $M$ infected individuals, is well approximated by the following:
\begin{align*}\label{e}
H(t) =\left(\tilde{F}(t)\right)^M.
\end{align*}
\end{theorem}

\section{Several examples of random function $\lambda(t)$}

Our varying infectivity model is in fact a SIR/SEIR, in the sense that it allows an exposed period just after infection, during which $\lambda(t)=0$. However, we do not introduce the $E$ compartment ($E$ for {\it Exposed}, the status of an infected individual who is, just after being infected, in a latent period, not yet infectious), the $I$ compartment  including all infected individuals, whether latent or infectious. In all most used models $\lambda(t)$ is piecewise constant, the jump times being random, following most classically an exponential distribution so that the stochastic model is Markovian and its law of large numbers limit is a system of ordinary differential equations (in contrast with the integral equation \eqref{AB}).

We now review two classical examples of piecewise constant $\lambda(t)$, which correspond respectively to the SIR and the SEIR model and finally present the example of varying infectivity $\lambda(t)$ which we shall use in the next section for our comparison with the more classical SIR ODE model.   

\subsection{The classical SIR model} 
The simplest commonly used example of the infectivity $\lambda(t)$ is  $\lambda(t) = \lambda\mathds{1}_{t \leq \eta}$, where $\lambda$ is a positive constant and $\eta$ is the random duration of the infectious period. In that case equation \eqref{ef1} take the form
\begin{equation*}\label{sirvi}
F(t) =\int_{0}^{t}\exp\Bigg\{\lambda\int_{0}^{r}\Big(F(t-u)-1\Big) du\Bigg\}\mathbb{P}_{\eta}(dr).
\end{equation*}
In the particular case of a deterministic $\eta$ ($i.e.$ $\P_\eta=\delta_a$, with $a\in \R_+$), we have 
\begin{align*}
F(t)
&=
\mathds{1}_{t \geq a}\exp\left\{\lambda\int_{0}^{a}\Big(F(t-u)-1\Big) du\right\} \quad \mbox{with} \quad F(0) = 0 \quad \mbox{and} \quad F(a) = \exp(-\lambda a).
\end{align*}
The most commonly used model corresponds to $\eta$ following an exponential distribution with parameter $\mu$. In this case, the system of  integral equations \eqref{AB} simplifies as follows :
 \begin{equation*}
 \left\{
    \begin{array}{ll}
  \frac{dS(t)}{dt}= -\lambda \overline{S}(t) I(t),   &\\\\ 
\frac{dI(t)}{dt}= \left(\lambda \overline{S}(t) - \mu \right) I(t), &\\\\ 
   \frac{dR(t)}{dt}=  \mu I(t).
&
           \end{array}
           \right.
\end{equation*}
If we linearize the second equation for $t\ge t_0$ by replacing $\overline{S}(t)$ by $\overline{S}(t_0)$, we obtain
\begin{equation*}
I(t)= I(t_0)\exp\left[ \left(\lambda  \overline{S}(t_0)- \mu\right)(t-t_0) \right]. 
\end{equation*}
From this, it is easy see that 
\begin{equation}\label{rhoODE}
\rho= \lambda  \overline{S}(t_0) - \mu.
\end{equation}
The fact that the above derivation is correct, although the deterministic model is
not valid for $t\ge t_0$, is explained in \cite{RPE}. Note also that solving equation \eqref{rhosolu} below gives the same result, as the reader can easily verify. 

Let us now compute $R_{eff}$. An infected individual has infectious contacts at rate $\lambda \overline{S}(t_0)$. This means that the expected number of infectious contacts equals
\begin{equation}\label{ReffODE}
R_{eff} =\lambda \overline{S}(t_0)\times \mathbb{E}[\eta] = \frac{\lambda  \overline{S}(t_0)}{\mu}.
\end{equation}
 The approximating branching process is the continuous time Markov branching process $(X(t))_{t\ge 0}$ which describes the number of descendants alive at time $t$ of a unique ancestor born at time zero. Every individual in this population, independently of the others, lives for an exponential time with parameter $\mu$, and during its lifetime it gives births at rate $\lambda \overline{S}(t_0)$. His descendants reproduce according to the same procedure. We consider the subcritical case $\mu>\lambda \overline{S}(t_0)$. Let  $G(s,t) = \mathbb{E}\left(s^{X(t)} \right)$, $|s|\le1$, be the probability generating function of $X(t)$. On page 109 of Athreya and Ney \cite{AtNe}, or in formula (5) of Iwasa, Nowak, and Michor \cite{IW}, we find the explicit form :
\begin{equation*}
G(s,t)=\frac{ \mu(s-1)-e^{-\rho t}(\lambda \overline{S}(t_0) s-\mu) }{\lambda\overline{S}(t_0)(s-1) -e^{-\rho t}(\lambda \overline{S}(t_0) s- \mu)}.
\end{equation*}
where $\rho$ was defined in \eqref{rhoODE}. Let us define $T_{ext} = \inf\{t > 0 : X(t) = 0\}$. We notice that $F(t)=G(0,t) = \mathbb{P}(X_t = 0)=\mathbb{P}(T_{ext} \leq t)$ is the distribution function of the extinction time. From the expression for $G(s,t)$, we deduce the value of $F(t)$. 
\begin{proposition}\label{prodep}
When starting with a single ancestor at time $0$, the distribution function of the extinction time is given as :
\begin{align*}
F(t)=\frac{1-e^{\rho t}}{1 - R_{eff} \times e^{\rho t} },
\end{align*}
where $R_{eff}$ was defined in \eqref{ReffODE}.
\end{proposition}

\subsection{The classical SEIR model} 
In this model, upon infection an individual is first exposed (compartment $E$), during a period $\xi$, during which the individual is not infectious, then he becomes infectious and stays so for a duration $\eta$, during which he infects susceptibles at rate $\lambda$, and then finally recovers. In that case $\lambda(t)$ is $\lambda(t)=\lambda {\bf1}_{\xi \le t<\xi+\eta}$ and equation \eqref{ef1} takes the form
\begin{align*}
F(t) &= \int_{0}^{t}\int_{0}^{t-r}\exp\Bigg\{\lambda\int_{s}^{s+r}\Big(F(t-u)-1\Big)du\Bigg\}\P_{(\xi,\eta)}(ds,dr).
\end{align*}
When $\xi$ and $\eta$ are deterministic, that is to say $\P_{(\xi,\eta)}(ds,dr) = \delta_{a}(ds) \delta_{b}(dr)$, with $(a,b)\in \R_+^2$, we have
\begin{align*}
F(t) =\mathds{1}_{t \geq a+b}\exp\Bigg\{\lambda\int_{a}^{a+b}\Big(F(t-u)-1\Big) du\Bigg\}, \quad \mbox{with} \quad F(u) = 0, \quad \mbox{for all} \ u \in [0,a]. 
\end{align*}
 In case $\xi$ and $\eta$ are independent and follow exponential distributions with parameters resp. $\gamma$ and $\mu$, the deterministic model obeys the ODE  
 \begin{equation*}
 \left\{
    \begin{array}{ll}
  \frac{d\overline{S}(t)}{dt}= -\lambda \overline{S}(t) \overline{I}(t),   &\\\\ 
  \frac{d\overline{E}(t)}{dt}= \lambda \overline{S}(t) \overline{I}(t) - \gamma \overline{E}(t), &\\\\ 
\frac{d\overline{I}(t)}{dt}= \gamma  \overline{E}(t) - \mu \overline{I}(t), &\\\\ 
   \frac{d\overline{R}(t)}{dt}=  \mu  \overline{I}(t).
&
           \end{array}
           \right.
\end{equation*}
In this model, again $R_{eff} = \frac{\lambda  \overline{S}(t_0)}{\mu}$. Solving the equation \eqref{rhosolu} below for $\rho$, we find 
\begin{equation}\label{rhoseir}
\rho= \frac{1}{2} \left[ \sqrt{(\gamma-\mu)^2+4\gamma  \overline{S}(t_0)\lambda}-(\mu+\gamma) \right].
\end{equation}

\subsection{Our varying infectivity model}
We again  define $\hat{\lambda}(t)= \overline{S}(t_0) \lambda(t)$. The infectivity $\hat{\lambda}(t)$ is first zero (corresponding to the latency period) followed by a gradual increase for some days, and then $\hat{\lambda}(t)$ starts decreasing down towards $0$ which it hits when the individual has recovered (see Figure 1). 
\begin{figure}[H]
\centering
\includegraphics[width=4.0in]{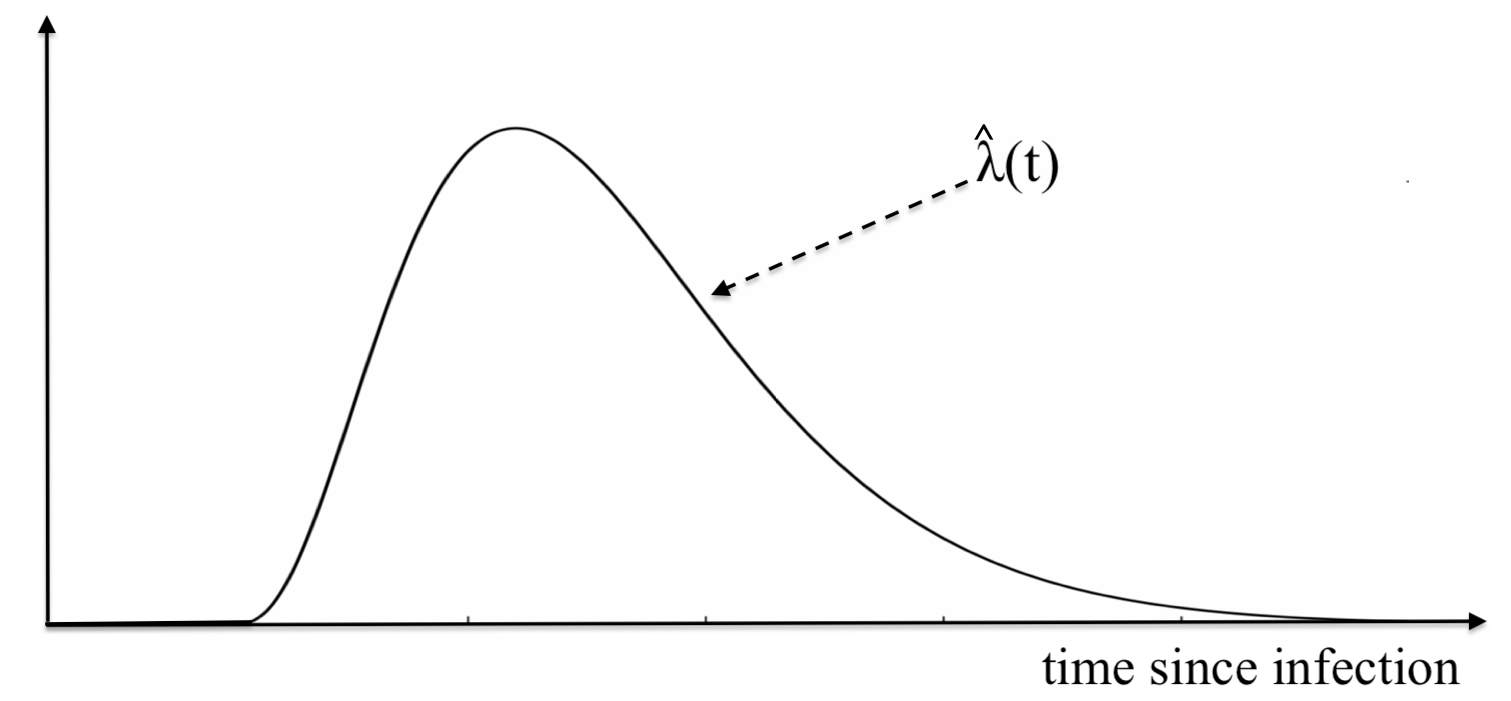}
\caption{Example of trajectory of $\hat{\lambda}(t)$.}
\end{figure}  
In the computations of section 5 below, we use a piecewise linear $\hat{\lambda}(t)$, which allows the function to depend upon a  small number of parameters, see Figure 2.
\begin{figure}[H]
\centering
\includegraphics[width=4.0in]{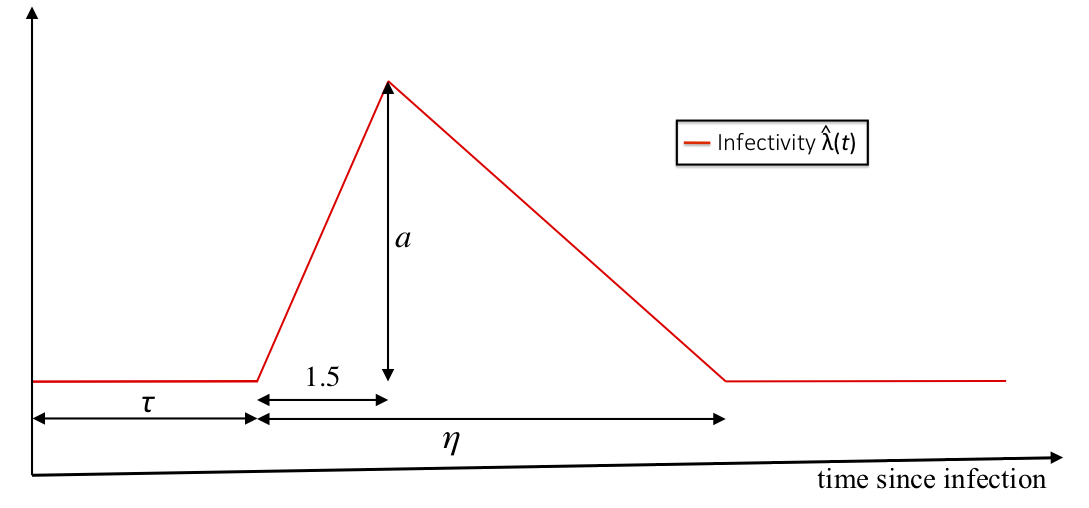}
\caption{trajectory of $\hat{\lambda}(t)$ used for the comparisons below.}
\end{figure}  
Here $\tau$ is the duration of the exposed period, $\eta$ that of the infectious period. We have arbitrarily fixed the length of the period of increase to $1.5$ days, and taken the maximum value $a$ to be a deterministic quantity at our disposal. In other word, in this case, we have  
\begin{equation}\label{lmdat}
\hat{\lambda}(t) =
\begin{cases}
0 & \text{if}~ t < \tau \\
\frac{a}{1.5}\left(t-\tau\right)  & \text{if}~ \tau \le t< \tau+1.5\\
a \frac{\tau+\eta-t}{\eta-1.5}  & \text{if}~ \tau+1.5 \le t < \tau+\eta\\
0  & \text{if}~ \tau+\eta < t 
\end{cases}
\end{equation}
Let $\mathcal{J}$ be the joint law of $\tau$ and $\eta$. From Corollary \ref{cor1}, we deduce that
\begin{align*}
F(t) &= \mathbb{E}\Bigg[ \mathds{1}_{\zeta \leq t}\exp\Bigg\{\frac{a}{1.5}\int_{\tau}^{\tau+1.5}(F(t-u)-1)(u-\tau)du
+\frac{a}{\eta-1.5}\int_{\tau+1.5}^{\tau+\eta }(F(t-u)-1)(\tau+\eta-u)du\Bigg\}\Bigg],
\end{align*}
with $\zeta = \tau+\eta$. Thus, we obtain
\begin{align*}
F(t) =& \int_{0}^{t}\int_{0}^{t} \mathds{1}_{s+r \leq t}\exp\Bigg\{ \frac{a}{1.5}\int_{s }^{s+1.5}(F(t-u)-1)(u-s)du\\
&+\frac{a}{r-1.5}\int_{s+1.5}^{s+r}(F(t-u)-1)(s+r-u)du\Bigg\}\mathcal{J}(ds,dr).
\end{align*}
The effective reproduction number is defined by
\begin{equation}\label{Reff=}
R_{eff}= \mathbb{E}\left[\int_{0}^{\infty}\hat{\lambda}(t)dt\right]
\end{equation}
and the rate of decrease $\rho$ of the number of infected individuals is the unique solution of 
\begin{equation}\label{rhosolu}
\mathbb{E}\left[\displaystyle\int_{0}^{\infty}e^{-\rho t}\hat{\lambda}(t)dt \right] =1,
\end{equation}
see Theorem 2.3 in \cite{RPE}.

\section{Comparison between our Varying infectivity model and an ODE SIR model}
 In this section, we compare the distribution function of the extinction time in our varying infectivity model, with that of an ODE SIR model with the same $R_{eff}$, which is the effective reproduction number at time $t_0$, and the same rate of decrease $\rho$ of the number of infected individuals. 

In the following, we assume that the random variables $\tau$ and $\eta$ defined in \eqref{lmdat} are independent, $\tau \sim \mathcal{U}\left(1.5, 2.5\right)$ and $\eta \sim \mathcal{U}\left(7,13\right)$.
\subsection{Approximation of the distribution function of the extinction time in the varying infectivity model}
Since it is not possible to obtain an explicit solution of \eqref{ef1}, then we will use the approximation made in section \ref{Appendix}. In other words, we will consider the following approximate function (whose convergence is established in section \ref{Appendix} below):
\begin{align*}
F_n\left(\frac{k}{n}\right) &= \mathbb{E}\left[\mathds{1}_{\tau+\eta \leq \frac{k}{n}}\exp\left\{\sum_{\ell=1}^{k}\left(F_n\left(\frac{k-\ell}{n}\right)-1\right)\int_{\frac{\ell-1}{n}}^{\frac{\ell}{n}}\hat{\lambda}(u)du\right\}\right].
\end{align*}
Let us define $\xi_{n,\ell} = \displaystyle\int_{\frac{\ell-1}{n}}^{\frac{\ell}{n}}\hat{\lambda}(u)du$. It is easy to see that $\xi_{n,\ell} \approx \frac{\hat{\lambda}\left(\frac{\ell}{n}\right)}{n}$. Combining this with \eqref{lmdat}, we deduce that 
\begin{align*}
\xi_{n,\ell} \approx  \frac{a}{1.5}\left(\frac{\ell}{n}-\tau\right)\mathds{1}_{\tau \le \frac{\ell}{n} < \tau+1.5} + a \left( \frac{\tau+\eta-\frac{\ell}{n}}{\eta-1.5} \right)\mathds{1}_{\tau+1.5 \le \frac{\ell}{n} < \tau+\eta}.
\end{align*}
Now, using the fact that the random variables $\tau$ and $\eta$ are independent, $\tau \sim \mathcal{U}\left(1.5;2.5\right)$ and $\eta \sim \mathcal{U}\left(7;13\right)$, we deduce that 
\begin{align}\label{fnkapp}
F_n\left(\frac{k}{n}\right)&\approx  \frac{1}{6}\int_{1.5}^{2.5}\int_{7}^{13}\mathds{1}_{x+y \leq \frac{k}{n}}\exp\left\{
\sum_{\ell=1}^{k}\left(F_n\left(\frac{k-\ell}{n}\right)-1\right)\frac{a}{1.5}\left(\frac{\ell}{n}-x\right)\mathds{1}_{x \le \frac{\ell}{n} < x+1.5} \right\} \nonumber \\
&\quad\times \exp\left\{\sum_{\ell=1}^{k}\left(F_n\left(\frac{k-\ell}{n}\right)-1\right)a \frac{x+y-\frac{\ell}{n}}{y-1.5}\mathds{1}_{x+1.5 \le \frac{\ell}{n} < x+y} \right\}dxdy \nonumber
\\
&\approx  \frac{1}{6}\frac{1}{n^2}\sum_{j=7n}^{13n}\sum_{i=1.5n}^{2.5n}\mathds{1}_{i+j \leq k}\exp\left\{
\sum_{\ell=i}^{i+1.5n}\left(F_n\left(\frac{k-\ell}{n}\right)-1\right)\left(\ell-i\right)\frac{a}{1.5n^2}\right\} \nonumber \\
&\quad\times \exp\left\{\sum_{\ell=i+1.5n}^{i+j}\left(F_n\left(\frac{k-\ell}{n}\right)-1\right) \frac{i+j-\ell}{j-1.5n}\frac{a}{n} \right\}.
\end{align}
\subsection{Computation of $R_{eff}$}
Recall \eqref{Reff=}. We first compute the random quantity $\displaystyle\int_{0}^{\infty}\hat{\lambda}(t)dt$. This is the surface below the curve $\hat{\lambda}(t)$, $i.e.$ the surface of the union of two triangles, and $\displaystyle\int_{0}^{\infty}\hat{\lambda}(t)dt = \frac{a\eta}{2}$.

Therefore, we have
\begin{align*}
R_{eff} = \frac{a}{2}\mathbb{E}[\eta] =\frac{a}{2}\times 10 = 5a.
\end{align*}

\subsection{Resolution of equation \eqref{rhosolu}}

From (\ref{lmdat}), we have
\begin{equation*}
\mathbb{E}\left[\int_{0}^{\infty}e^{-\rho t}\hat{\lambda}(t)dt \right]= a\left(A_\rho+B_\rho\right), \quad \mbox{with}
\end{equation*}
\begin{equation*}
A_\rho= \E\left(\int_{\tau}^{\tau+1.5}e^{-\rho t}\frac{t-\tau}{1.5}dt \right) \quad \mbox{and} \quad B_\rho= \E\left(\int_{\tau+1.5}^{\tau+\eta}e^{-\rho t}\frac{\tau+\eta-t}{\eta-1.5}dt \right).
\end{equation*}
Using the fact that $\tau$ and $\eta$ are independent, $\tau \sim \mathcal{U}\left(1.5;2.5\right)$, $\eta \sim \mathcal{U}\left(7;13\right)$, it is easy to check that
\begin{equation*}
A_\rho = \frac{1}{\rho}\left(e^{-1.5\rho}-e^{-2.5\rho}\right)\left[\frac{1}{1.5\rho^2}-e^{-1.5\rho}\left(\frac{1}{\rho}+\frac{1}{1.5\rho^2}\right)\right], \quad \mbox{and}
\end{equation*}
\begin{equation*}
B_\rho=\frac{1}{\rho}\left(e^{-1.5\rho}-e^{-2.5\rho}\right)\left\{ e^{-1.5\rho}\left(\frac{1}{\rho}-\frac{1}{6\rho^2}\log\left(\frac{11.5}{5.5}\right)\right)+\frac{1}{\rho^2}\mathbb{E}\left[\frac{e^{-\rho\eta}}{(\eta-1.5)}\right] \right\}.
\end{equation*}
Note that the mapping $\rho\mapsto\mathbb{E}\int_{0}^{\infty}e^{-\rho t}\hat{\lambda}(t)dt$ is decreasing. Consequently, it is easy to compute an approximate solution of equation \eqref{rhosolu}.
\subsection{Comparison of the distributions and the expectations of the extinction time between our Varying infectivity model and a ODE SIR model}\label{section54}
In what follows, we compare the extinction time in our Varying infectivity model and in the ODE SIR model with the same $R_{eff}$ and $\rho$. Note that we compare $F$'s and not $H$'s (see the notations in section \ref{sec3}). We compare the distribution of the extinction time of our Varying infectivity model given in \eqref{fnkapp} and of the extinction time of the ODE SIR model given in Proposition \ref{prodep}.
\begin{figure}[H]
	\centering
	\includegraphics[width=7.5cm]{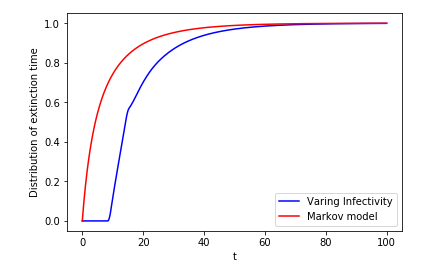}
	\caption{Comparison of models with the same $R_{eff} = 0.66$ and  $\rho =  -0.0683$.}
	\end{figure}
\begin{figure}[H]
	\centering
	\includegraphics[width=7.5cm]{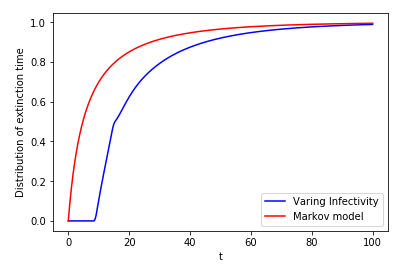}
	\caption{Comparison of models with the same  $R_{eff} = 0.8$ and  $\rho = -0.03816$.}
	\label{}
\end{figure}

We also compare the expectations of the extinction times of our varying infectivity model and of an ODE SIR model. To this end, recall that, the extinction time can be rewrite in the form $T_{ext}= \inf\{t-t_0 \ : \ I(t-t_0)=0 \}$. Thus, for the ODE SIR model, we obtain  
\begin{align*}
\mathbb{E}[T_{ext}] = \int_{0}^{\infty}\mathbb{P}(T_{ext}>t)dt = \int_{0}^{\infty}\left(1-F(t)\right)dt =  \frac{(1-R_{eff})}{\rho R_{eff}}\ln(1-R_{eff}),
\end{align*}
where we have used the formula of Proposition \ref{prodep} for $F(t)$.

For the varying infectivity model, we obtain 
\begin{align*}
\mathbb{E}[T_{ext}] = \int_{0}^{\infty}\mathbb{P}(T_{ext}>t)dt = \int_{0}^{\infty}\left(1-F_n(t)\right)dt \approx \frac{1}{n}\sum_{k = 1}^{n\Lambda}\left(1-F_n\left(\frac{k}{n}\right)\right).
\end{align*}
where $\Lambda$ is the point where we stop the calculation of the integral of $1-F_n(t)$. 

 \[\begin{tabular}{|*{11}{c|}}
 \hline
     &  &    \\
  & $R_{eff}= 0.66$ &  $R_{eff}= 0.8$   \\
     &  &   \\
      \hline
        &  &    \\
   & $\rho = -0.0683$ & $\rho = -0.03816$   \\
     &  &   \\
    \hline
        &  &   \\
  \mbox{ Varying infectivity model }  & $\mathbb{E}[T_{ext}] \approx 18.7854$ & $\mathbb{E}[T_{ext}] \approx 22.6568$   \\
     &  &   \\
    \hline
      &  &    \\
   \mbox{ODE SIR model}  & $\mathbb{E}[T_{ext}] =8.1369$ & $\mathbb{E}[T_{ext}] =10.544$ \\
       &  &  \\
  \hline
\end{tabular}
\]
\section{Conclusion :} Our comparison shows that, in the final phase of the epidemic, the varying infectivity SIR model (in fact its branching process approximation) tends to take more time to extinct then the branching process approximation of the ODE SIR model. This is not too surprising, since the varying infectivity model has a memory, contrary to the ODE SIR model. This fact is easily seen when there is a sudden change of the propagation of the epidemic like the lockdown that several countries have established during recent Covid epidemic. The authors who use an ODE model change gradually the infection rate, starting with the lockdown, while in reality the change of the infection rate was very sudden. This is a way to compensate the lack of memory of ODE models. We believe that the fact the varying infectivity SIR model takes more time than the ODE SIR model to forget its past explains that it takes more time to go extinct. The varying infectivity SIR model is more complex that the more classical ODE SIR model, and this probably explains why most authors who quote the seminal $1927$ paper of Kermack and McKendrick \cite{KMK} refer only to the very particular case of constant coefficients, studied in section 3.2 of that paper. Of course, it is very tempting and sometime preferable to use simple models, which allow to draw more conclusions. However, it is crucial to understand which biais the simple model introduce, compared to more realistic models. In this paper, we have identified one of those biaises, namely the shortening of the final phase of the epidemic.  In future work, we intend to do similar computations with 
various classes of varying infectivity models, in order to confirm these first conclusions.

\section{Appendix : Approximation of the distribution function of the extinction time}\label{Appendix} 
	In this subsection, we define a sequence of functions $\{F_n, n \geq 1\} $ which will allow us to approach the solution of equation \eqref{ef1}. To this end, for each $k \in \mathbb{Z}_+$, we set 
	\begin{equation}\label{Fnk}
	F_n\left(\frac{k}{n}\right) = \mathbb{E}\Bigg[\mathds{1}_{\eta \leq \frac{k}{n}}\exp\Bigg\{\sum_{\ell=1}^{k}\left(F_n\left(\frac{k-\ell}{n}\right)-1\right)\int_{\frac{\ell-1}{n}}^{\frac{\ell}{n}}\lambda(u)du\Bigg\}\Bigg]
	\end{equation}
	and for each $t \in [\frac{k}{n}, \frac{k+1}{n})$, 
	\begin{equation}\label{Fnt}
	F_n(t) = \mathbb{E}\Bigg[\mathds{1}_{\eta \leq t}\exp\Bigg\{\sum_{\ell=1}^{k-1}\left(F_n\left(\frac{k-\ell}{n}\right)-1\right)\int_{\frac{\ell-1}{n}}^{\frac{\ell}{n}}\lambda(u)du-\int_{\frac{k-1}{n}}^{t}\lambda(u)du\Bigg\}\Bigg].
	\end{equation}
	The goal of this section is to prove that as $n\longrightarrow +\infty$, $\{F_n(t), t>0\} \longrightarrow \{F(t), t>0\}$ in $D([0, +\infty))$, where $F$ is the unique solution of \eqref{ef1}. 
	
	We first check that 
	\begin{lemma}\label{lem1}
		For any $k \in \mathbb{Z}_+$, we have 
		\begin{equation*}
		F_n\left(\frac{k}{n} \right) \le F\left(\frac{k}{n}\right) \le 1.
		\end{equation*}
	\end{lemma}
	\begin{proof} 
		Let $k \in \mathbb{Z}_+$. We first note that $F(t)\le 1$ (since $F$ is a distribution function). To prove the next assertion, we will proceed by recurrence on $k$. It is clear that $F_n(0)=0$. Let us now suppose that $F_n \left(\frac{\ell}{n}\right) \leq F\left(\frac{\ell}{n}\right)$, $\forall 1\le \ell \le k-1$. Now let us show that $F_n \left(\frac{k}{n}\right) \leq F\left(\frac{k}{n}\right)$. We have
		\begin{align*}
		F\left(\frac{k}{n}\right) &= 
		\mathbb{E}\left[\mathds{1}_{\eta \leq \frac{k}{n}}\exp\left\{\int_{0}^{\frac{k}{n}}\left(F\left(\frac{k}{n}-u\right)-1\right)\lambda(u)du\right\}\right]\\
		&\geq
		\mathbb{E}\left[\mathds{1}_{\eta \leq \frac{k}{n}}\exp\left\{\sum_{\ell=1}^{k}\int_{\frac{\ell-1}{n}}^{\frac{\ell}{n}}\left(F\left(\frac{k-\ell}{n}\right)-1\right)\lambda(u)du\right\}\right]\\
		&\geq
		\mathbb{E}\left[\mathds{1}_{\eta \leq \frac{k}{n}}\exp\left\{\sum_{\ell=1}^{k}\left(F_n\left(\frac{k-\ell}{n}\right)-1\right)\int_{\frac{\ell-1}{n}}^{\frac{\ell}{n}}\lambda(u)du\right\}\right]\\
		&=F_n\left(\frac{k}{n}\right),
		\end{align*}
		where we have used the fact that $F$ is non-decreasing and the recurrence assumption. 
	\end{proof} 
	
The previous extends to all $t$.
	\begin{lemma}\label{lem2}
		For any $t \ge 0$, we have 
		\begin{equation*}
		F_n\left(t \right) \le F\left(t\right) \le 1.
		\end{equation*}
	\end{lemma}
	\begin{proof} 
		We first note that
		\begin{equation*}
		\int_{0}^{t}\left(1-F(t-u)\right)\lambda(u) du \leq \sum_{\ell=1}^{[ nt]}\int_{\frac{\ell-1}{n}}^{\frac{\ell}{n}}\left(1-F(t-u)\right)\lambda(u)du+\int_{\frac{[nt]}{n}}^{t}\lambda(u)du.
		\end{equation*}
		From the fact that $F$ is non-decreasing and $\frac{[nt]}{n} \leq t$, we deduce that
		\begin{align*}
		\int_{0}^{t}\left(1-F(t-u)\right)du &\leq \sum_{\ell=1}^{\lceil nt\rceil}\int_{\frac{\ell-1}{n}}^{\frac{\ell}{n}}\left(1-F\left(\frac{[nt]-\ell}{n}\right)\right)\lambda(u)du+\int_{\frac{[nt]}{n}}^{t}\lambda(u)du\\
		\int_{0}^{t}\left(F(t-u)-1\right)du 
		&\geq \sum_{\ell=1}^{[nt]}\left(F_n\left(\frac{\lceil nt\rceil-\ell}{n}\right)-1\right)\int_{\frac{\ell-1}{n}}^{\frac{\ell}{n}}\lambda(u)du-\int_{\frac{[nt]}{n}}^{t}\lambda(u)du,
		\end{align*}
		where we have used Lemma \ref{lem1} for the last inequality. The desired result follows by combining the last inequality with \eqref{Fnt}.
	\end{proof}

	We have 
	\begin{proposition}\label{prdep}
		Let $T>0$. Then there exists a constant $C$ such that for all  $n\ge 1$ and $0 < s < t<T$, 
		\begin{equation*}
		-\frac{C}{n} -C(t-s) \le F_n(t)-F_n(s) \le C(t-s)+ \phi(t) -\phi(s)+ \frac{C}{n},
		\end{equation*}
		where $\phi(t)=\P(\eta \le t)$ the distribution function of $\eta$. 
	\end{proposition}
	
For the proof of this proposition, we will need some several technical lemmas. In order to simplify the notations below we let 
	\begin{equation}\label{akbk}
	a_n(k)= \left[F_n\left(\frac{k+1}{n}\right)-F_n\left(\frac{k}{n}\right)\right]^{-} \quad \mbox{and} \quad  b_n(k)= \left[F_n\left(\frac{k+1}{n}\right)-F_n\left(\frac{k}{n}\right)\right]^{+}.
	\end{equation}
	Let us define, $\forall n \geq 1$, $k\in \mathbb{Z}_+$, 
	\begin{equation}\label{Lamdank}
	\Lambda_n(k) = \sum_{\ell=1}^{k}\bigg(F_n\Big(\frac{k-\ell}{n}\Big)-1\bigg)\int_{\frac{\ell-1}{n}}^{\frac{\ell}{n}}\lambda(u)du\le 0, 
	\end{equation}
	(see Lemma \ref{lem1}) and let us rewrite \eqref{Fnk} in the form 
	\begin{equation}\label{Fnk1}
	F_n\Big(\frac{k}{n}\Big)= \mathbb{E}\left[\mathds{1}_{\eta \leq \frac{k}{n}}\exp(\Lambda_n(k))\right].
	\end{equation}
	We will need the following lemmas.
	\begin{lemma}\label{lemA1A2}
		For any $n\ge 1$, $k \in \mathbb{Z}_+$, we have 
		\begin{equation*}
		A_1(n,k) \le F_n\Big(\frac{k+1}{n}\Big) -F_n\Big(\frac{k}{n}\Big) \le  A_2(n,k),
		\end{equation*}
		with
		\begin{equation}\label{A1nk}
		A_1(n,k)= \exp\Bigg\{-\frac{\lambda^*}{n}\left[\sum_{\ell=0}^{k-1}a_n(\ell) +1\right]\Bigg\}-1
		\end{equation}
		and 
		\begin{equation}\label{A2nk}
		A_2(n,k)= \exp\Bigg\{\frac{\lambda^*}{n}\sum_{\ell=0}^{k-1} b_n(\ell) \Bigg\}-1 
		+ \mathbb{P}\Big(\frac{k}{n} < \eta \leq \frac{k+1}{n}\Big). 
		\end{equation}
	\end{lemma}
	\begin{proof} 
		Recalling \eqref{Lamdank} and \eqref{Fnk1}, we first note that 
		\begin{align*}
		\Lambda_n(k+1)-\Lambda_n(k) = \sum_{\ell=1}^{k}\Big(F_n\Big(\frac{k+1-\ell}{n}\Big)-F_n\Big(\frac{k-\ell}{n}\Big)\Big)\int_{\frac{\ell-1}{n}}^{\frac{\ell}{n}}\lambda(u)du-\int_{\frac{k}{n}}^{\frac{k+1}{n}}\lambda(u)du.
		\end{align*}
		It follows that
		\begin{align*}
		-\Big(\Lambda_n(k+1)-\Lambda_n(k)\Big)^{-} \geq -\sum_{\ell=1}^{k}\Big(F_n\Big(\frac{k+1-\ell}{n}\Big)-F_n\Big(\frac{k-\ell}{n}\Big)\Big)^{-}\int_{\frac{\ell-1}{n}}^{\frac{\ell}{n}}\lambda(u)du-\int_{\frac{k}{n}}^{\frac{k+1}{n}}\lambda(u)du
		\end{align*}
		and 
		\begin{align*}
		\Big(\Lambda_n(k+1)-\Lambda_n(k)\Big)^{+} \leq \sum_{\ell=1}^{k}\Big(F_n\Big(\frac{k+1-\ell}{n}\Big)-F_n\Big(\frac{k-\ell}{n}\Big)\Big)^{+}\int_{\frac{\ell-1}{n}}^{\frac{\ell}{n}}\lambda(u)du.
		\end{align*}
		Thus, we have
		\begin{align*}
		F_n\Big(\frac{k+1}{n}\Big)-F_n\Big(\frac{k}{n}\Big) &= \mathbb{E}\left[\mathds{1}_{\eta \leq \frac{k+1}{n}}\exp(\Lambda_n(k+1))
		-\mathds{1}_{\eta \leq \frac{k}{n}}\exp(\Lambda_n(k))\right]\\
		&= \mathbb{E}\left[\left(\mathds{1}_{\eta \leq \frac{k+1}{n}}-\mathds{1}_{\eta \leq \frac{k}{n}}\right)\exp(\Lambda_n(k+1))
		+\mathds{1}_{\eta \leq \frac{k}{n}}\left(\exp(\Lambda_n(k+1))-\exp(\Lambda_n(k))\right)\right]\\
		&\leq \mathbb{E}\left[\left(\mathds{1}_{\eta \leq \frac{k+1}{n}}-\mathds{1}_{\eta \leq \frac{k}{n}}\right)
		+\mathds{1}_{\eta \leq \frac{k}{n}}\exp(\Lambda_n(k))\left(\exp(\Lambda_n(k+1)-\Lambda_n(k))-1\right)\right]\\
		&\leq \P\left(\frac{k}{n} < \eta \leq \frac{k+1}{n} \right)+ \E\left(\exp\left[{\left(\Lambda_n(k+1)-  \Lambda_n(k)\right)^+}\right] -1\right)\\
		&\leq \mathbb{E}\left[\exp\Bigg\{\sum_{\ell=1}^{k}\Big(F_n\Big(\frac{k+1-\ell}{n}\Big)-F_n\Big(\frac{k-\ell}{n}\Big)\Big)^+\int_{\frac{\ell-1}{n}}^{\frac{\ell}{n}}\lambda(u)du\Bigg\}\right]-1\\
		& \quad +\mathbb{P}\Big(\frac{k}{n} < \eta \leq \frac{k+1}{n}\Big)\\
		&\leq A_2(n,k),
		\end{align*}
		where we have used \eqref{akbk} and \eqref{A2nk} in the last inequality. We also have 
		\begin{align*}
		F_n\Big(\frac{k+1}{n}\Big)-F_n\Big(\frac{k}{n}\Big) &= \mathbb{E}\Big[\mathds{1}_{\eta \leq \frac{k+1}{n}}\exp(\Lambda_n(k+1))-\mathds{1}_{\eta \leq \frac{k}{n}}\exp(\Lambda_n(k))\Big]\\
		&\geq  \mathbb{E}\Big[\mathds{1}_{\eta \leq \frac{k}{n}}\exp(\Lambda_{n}(k))\Big(\exp(\Lambda_n(k+1)-\Lambda_n(k))-1\Big)\Big]\\
		&\geq  \mathbb{E}\Big[\mathds{1}_{\eta \leq \frac{k}{n}}\exp(\Lambda_{n}(k))\Big(\exp(-(\Lambda_n(k+1)-\Lambda_n(k))^{-})-1\Big)\Big]\\
		& \geq \mathbb{E}\Big[\exp(-(\Lambda_n(k+1)-\Lambda_n(k))^{-})-1\Big].
		\end{align*}
		Combining the above arguments with \eqref{akbk} and \eqref{A1nk}, we deduce that
		\begin{align*}
		F_n\Big(\frac{k+1}{n}\Big)-F_n\Big(\frac{k}{n}\Big) \geq A_1(n,k).
		\end{align*}
	\end{proof}
	
	Recall \eqref{akbk}. We have 
	\begin{lemma}\label{abnkC}
		Let $T>0$. Then there exists a constant $C$ such that for all $n\ge1$ and $0\le \frac{k}{n}< T$,
		\begin{equation*}
		\sum_{\ell=0}^{k} a_n(\ell) \le C \quad \mbox{and} \quad \sum_{\ell=0}^{k} b_n(\ell) \le C.
		\end{equation*}
	\end{lemma}
	\begin{proof} 
		Let us show the first assertion. For this, we first prove that
		\begin{equation*}
		a_n(k)  \leq r\left(1+r\right)^{k-1},  \quad \text{with}\quad r = \frac{\lambda^*}{n}. 
		\end{equation*}
		According to Lemma \ref{lemA1A2}, we have
		\begin{align*}
		a_n(k)\le -A_1(n,k)&= 1-\exp\Bigg\{-r\left[\sum_{\ell=0}^{k-1}a_n(\ell) +1\right]\Bigg\}\\
		&\le r\left(\sum_{\ell=0}^{k-1}a_n(\ell) +1\right).
		\end{align*}
		However, it easy to see that $a_n(0)=0$ and $a_n(1)\le r$. Let us suppose $a_n(\ell)  \leq r\left(1+r\right)^{\ell-1}$, $\forall 1\le \ell \le k-1$. Thus, it is easy to see that, 
		\begin{align*}
		a_n(k) &\leq r \left(1+r+r(1+r)+...+r(r+1)^{k-2}\right)\\
		&= r\left(1+r\sum_{i=1}^{k-1}(1+r)^{i-1}\right)= r(1+r)^{k-1}.
		\end{align*}
		Consequently, since $\frac{k}{n} \leq T$,
		\begin{equation*}
		\sum_{\ell=0}^{k} a_n(\ell)= \sum_{\ell=1}^{k} a_n(\ell) \le  \sum_{\ell=1}^{k} r(1+r)^{\ell-1} =  (1+r)^k-1\le e^{rk}\le e^{\lambda^*T}\le C_T.
		\end{equation*}
		We now show the seond assertion. We first have that $b_n(0)=F_n\left(\frac{1}{n}\right)$. Then we have 
		\begin{align*}
		\sum_{\ell=0}^{k} b_n(\ell)&=F_n\left(\frac{1}{n}\right) +\sum_{\ell=1}^{k} \left(F_n\Big(\frac{\ell+1}{n}\Big)-F_n\Big(\frac{\ell}{n}\Big) \right) +\sum_{\ell=1}^{k} a_n(\ell)
		\\ &= F_n\Big(\frac{k+1}{n}\Big) +\sum_{\ell=1}^{k} a_n(\ell).
		\\ &\le 1 +\sum_{\ell=1}^{k} a_n(\ell),
		\end{align*}
		where we have used Lemma \ref{lem1} in the last inequality.  The desired result follows by combining this with the first assertion.
	\end{proof} 
	
	We shall need the following
	\begin{lemma}\label{lemlk}
		Let $T>0$. Then there exists a constant $C$ such that for all  $n\ge 1$ and $0\le \frac{\ell}{n}< \frac{k}{n}< T$, 
		\begin{equation*}
		-C\left(\frac{k-\ell}{n}\right) \le F_n\left(\frac{k}{n} \right)-F_n\left(\frac{\ell}{n} \right) \le C\left(\frac{k-\ell}{n}\right) + \phi\left(\frac{k}{n} \right) -\phi\left(\frac{\ell}{n} \right),
		\end{equation*}
		where $\phi(t)=\P(\eta \le t)$ the distribution function of the random variable $\eta$. 
	\end{lemma}
	\begin{proof} 
		Recall \eqref{A1nk} and \eqref{A2nk}. We have 
		\begin{align*}
		-A_1(n,k)&= 1-\exp\Bigg\{-\frac{\lambda^*}{n}\left[\sum_{\ell=0}^{k-1}a_n(\ell) +1\right]\Bigg\}
		\\&\le \frac{\lambda^*}{n}\left[\sum_{\ell=0}^{k-1}a_n(\ell) +1\right]\\
		&\le \frac{C}{n},
		\end{align*}
		where we have used Lemma \ref{abnkC}. However, we have 
		\begin{align*}
		A_2(n,k)&= \exp\Bigg\{\frac{\lambda^*}{n}\sum_{\ell=0}^{k-1} b_n(\ell) \Bigg\}-1 
		+ \mathbb{P}\Big(\frac{k}{n} < \eta \leq \frac{k+1}{n}\Big)\\
		&\le C\frac{\lambda^*}{n}\exp\left\{C\frac{\lambda^*}{n}\right\}+ \mathbb{P}\Big(\frac{k}{n} < \eta \leq \frac{k+1}{n}\Big)\\
		&\le \frac{C}{n}+ \mathbb{P}\Big(\frac{k}{n} < \eta \leq \frac{k+1}{n}\Big),
		\end{align*}
		where we have used the fact that $e^x-1\le xe^x$, $\forall x\ge0$ and Lemma \ref{abnkC}. Now combining the above arguments with Lemma \ref{lemA1A2}, we deduce that 
		\begin{equation*}
		-\frac{C}{n} \le F_n\Big(\frac{k+1}{n}\Big) -F_n\Big(\frac{k}{n}\Big) \le  \frac{C}{n}+ \mathbb{P}\Big(\frac{k}{n} < \eta \leq \frac{k+1}{n}\Big).
		\end{equation*}
		However, we note that 
		\begin{equation*}
		F_n\left(\frac{k}{n} \right)-F_n\left(\frac{\ell}{n} \right) = \sum_{j=\ell}^{k-1}\left( F_n\Big(\frac{j+1}{n}\Big) -F_n\Big(\frac{j}{n}\Big)\right).
		\end{equation*}
		The desired result follows by combining this with the previous inequalities.
	\end{proof} 
	
	Let us define, $\forall n \geq 1, t > 0$, with $k = \lceil nt\rceil$,
	\begin{equation}\label{Lamdankt}
	\Lambda_n(t)=\sum_{\ell=1}^{k-1}\left(F_n\left(\frac{k-\ell}{n}\right)-1\right)\int_{\frac{\ell-1}{n}}^{\frac{\ell}{n}}\lambda(u)du-\int_{\frac{k-1}{n}}^{t}\lambda(u)du,
	\end{equation}
	(see Lemma \ref{lem1}) and let us rewrite \eqref{Fnt} in the form 
	\begin{equation}\label{Fnt1}
	F_n(t) = \mathbb{E}\Big[\mathds{1}_{\eta \leq t}\exp\left(  \Lambda_n(t)\right) \Big].
	\end{equation}
	We shall need the following
	\begin{lemma}\label{touslamda}
		Let $T>0$. Then there exists a constant $C$ such that for all  $n\ge 1$ and $0<\frac{\ell-1}{n}<s< \frac{\ell}{n}< \frac{k}{n}< t< \frac{k+1}{n}<T$, 
		\begin{equation*}
		\left( \Lambda_n(t)-  \Lambda_n(k) \right)^+=0, \quad  \left( \Lambda_n(t)-  \Lambda_n(k) \right)^- \le C\left(t-\frac{k}{n}\right),
		\end{equation*}
		\begin{equation*}
		\left( \Lambda_n(\ell) -\Lambda_n(s) \right)^+\le \frac{C}{n} \quad \mbox{and} \quad \left( \Lambda_n(\ell) -\Lambda_n(s) \right)^-\le \frac{C}{n}+ C\left(\frac{\ell}{n}-s\right),
		\end{equation*}
		where $ \Lambda_n(.)$ was defined in \eqref{Lamdank}.
	\end{lemma}
	\begin{proof} 
		From \eqref{Lamdank} and \eqref{Lamdankt}, we have 
		\begin{equation*}
		-\lambda^*\left(t-\frac{k}{n}\right) \leq \Lambda_n(t)- \Lambda_n(k)= - \int_{\frac{k}{n}}^{t}\lambda(u)du \le 0. 
		\end{equation*}
		Thus, we obtain the first two assertions. In the same way, from \eqref{Lamdank} and \eqref{Lamdankt} we have
		\begin{align*}
		\Lambda_n(\ell)-  \Lambda_n(s)=& \sum_{j=1}^{\ell-2}\left(F_n\left(\frac{\ell-j}{n}\right)-F_n\left(\frac{\ell-1-j}{n}\right) \right)\int_{\frac{j-1}{n}}^{\frac{j}{n}}\lambda(u)du-\int_{\frac{\ell-1}{n}}^{\frac{\ell}{n}}\lambda(u)du \\
		&+\left(F_n\left(\frac{1}{n}\right)-1\right) \int_{\frac{\ell-2}{n}}^{\frac{\ell-1}{n}}\lambda(u)du+ \int_{\frac{\ell-2}{n}}^{s}\lambda(u)du\\
		=&\sum_{j=1}^{\ell-2}\left(F_n\left(\frac{\ell-j}{n}\right)-F_n\left(\frac{\ell-1-j}{n}\right) \right)\int_{\frac{j-1}{n}}^{\frac{j}{n}}\lambda(u)du+ F_n\left(\frac{1}{n}\right) \int_{\frac{\ell-2}{n}}^{\frac{\ell-1}{n}}\lambda(u)du
		\\&- \int_{s}^{\frac{\ell}{n}}\lambda(u)du. 
		\end{align*}
		Combining this with Lemmas \ref{lem1}, \ref{abnkC} and \eqref{akbk}, we deduce that 
		\begin{equation*}
		\left( \Lambda_n(\ell) -\Lambda_n(\ell-1,s) \right)^+\le \frac{C}{n} \quad \mbox{and} \quad \left( \Lambda_n(\ell) -\Lambda_n(\ell-1,s) \right)^-\le \frac{C}{n}+ C\left(\frac{\ell}{n}-s\right),
		\end{equation*}
	\end{proof} 
	
	\begin{lemma}\label{lemtlks}
		Let $T>0$. Then there exists a constant $C$ such that for all  $n\ge 1$ and $0<\frac{\ell-1}{n}<s< \frac{\ell}{n}< \frac{k}{n}< t<\frac{k+1}{n}<T$, 
		\begin{equation*}
		-\frac{C}{n}-C\left(\frac{\ell}{n}-s\right) \le F_n\left(\frac{\ell}{n} \right)- F_n(s) \le \frac{C}{n} + \phi\left(\frac{\ell}{n} \right)- \phi\left(s \right)
		\end{equation*}
		and
		\begin{equation*}
		-C\left(t-\frac{k}{n}\right) \le F_n\left(t \right) - F_n\left(\frac{k}{n} \right) \le C\left(t-\frac{k}{n}\right) + \phi\left(t \right) -\phi\left(\frac{k}{n} \right),
		\end{equation*}
		where $\phi(t)=\P(\eta \le t)$ the distribution function of $\eta$. 
	\end{lemma}
	\begin{proof} 
		Recall \eqref{Fnk1} and \eqref{Fnt1}. From an easy adaptation of the argument of the proof of Lemma \ref{lemA1A2} and from Lemma \ref{touslamda}, we have that 
		\begin{align*}
		\E\left(e^{[\Lambda_n(\ell)-  \Lambda_n(\ell-1,s)]^-} -1\right)&\le F_n\left(\frac{\ell}{n} \right)- F_n(s) \le \P\left(s<\eta \le \frac{\ell}{n} \right)+ \E\left(e^{[\Lambda_n(\ell)-  \Lambda_n(\ell-1,s)]^+} -1\right)
		\\- \E\left([\Lambda_n(\ell)-  \Lambda_n(\ell-1,s)]^-\right) &\le F_n\left(\frac{\ell}{n} \right)- F_n(s) \le \phi\left(\frac{\ell}{n} \right)- \phi\left(s \right)+ C\E\left([\Lambda_n(\ell)-  \Lambda_n(\ell-1,s)]^+ \right)\\
		-\frac{C}{n}-C\left(\frac{\ell}{n}-s\right) &\le F_n\left(\frac{\ell}{n} \right)- F_n(s) \le \frac{C}{n} +\phi\left(\frac{\ell}{n} \right)- \phi\left(s \right).
		\end{align*}
		In the same way, we get the other assertion.
	\end{proof} 
	
	We can now turn to the
	
	$\bf{Proof \ of \  Proposition}$  \ref{prdep} :   By combining Lemmas \ref{lemlk}, \ref{lemtlks} and the fact that 
	\begin{equation*}
	F_n(t)-F_n(s) = F_n(t)-F_n\left(\frac{k}{n} \right)+F_n \left(\frac{k}{n} \right)-F_n\Big(\frac{\ell}{n}\Big)+F_n\Big(\frac{\ell}{n}\Big)-F_n(s),
	\end{equation*}
	we deduce that 
	\begin{align*}
	-\frac{C}{n} -C(t-s) &\le F_n(t)-F_n(s) \le  C\left(t-\frac{\ell}{n}\right)+ \phi(t) -\phi(s)+ \frac{C}{n}.
	\end{align*}
	It follows that
	\begin{align*}
	-\frac{C}{n} -C(t-s) &\le F_n(t)-F_n(s) \le C(t-s)+ \phi(t) -\phi(s)+ \frac{C}{n}.
	\end{align*}
	The desired result follows
	$\hfill \blacksquare$  
	
	Recall that the goal of this subsection is to prove the convergence of the sequence $(F_n)_{n\ge1}$ towards $F$, the unique solution of equation \eqref{ef1}. For $T>0$, we define $w_{T}^\prime(x,.)$ the modulus of continuity of $x$ $\in$ $D([0, +\infty))$ on the interval $[0,T]$ by
	\begin{equation*}
	{w}_{T}^\prime(x,\delta)= \inf \ \max_{0\le i<m}\sup_{{ t_i\leq s< t\leq t_{i+1}}} |x(t)-x(s)|, 
	\end{equation*}
	where the infimum is taken over the set of all increasing sequences  $0=t_0<t_1< \cdot \cdot \cdot <t_m= T $   with the property that $\inf_{0\leq i< m}| t_{i+1}-t_i|\geq \delta$.  Let $\{x_n, n\ge1 \}$ be a sequence function in $D([0, +\infty))$. The following result is a version of Theorem 12.3 from  \cite{Bill} :
	\begin{proposition}
		Let $T>0$. A necessary and sufficient condition for the sequence $\{x_n, n\ge 1\}$ to be relatively compact in $D([0, +\infty))$ is that 
		\begin{equation*}
		\quad \quad (i)\ \sup_{n\ge 1} \sup_{0\le t\le T} |x_n(t)|<+\infty
		\end{equation*}
		and
		\begin{equation*}
		\quad \quad (ii)\  \lim_{\delta\rightarrow 0} \limsup_{n\rightarrow +\infty} {w}_{T}^\prime(x_n,\delta)=0.
		\end{equation*}
	\end{proposition}
	We now show that the sequence $(F_n)_{n\ge1}$ satisfies the assertions of the above Proposition.  
	\begin{proposition}\label{Fncompact}
		The sequence $(F_n)_{n\ge1}$ is relatively compact in $D([0, +\infty))$.
	\end{proposition}
	\begin{proof}
		Condition $(i)$ follows from Lemma \ref{lem2}. Hence, it suffices to verify $(ii)$. To this end, let us define $\psi(t)=\phi(t)+Ct$, $\forall t>0$, where again $\phi$ is the distribution function on $\eta$. It follows from Proposition \ref{prdep}  that 
		\begin{equation}\label{FnsFnt}
		|F_n(t)- F_n(s)|\le \psi(t)-\psi(s)+\frac{C}{n}, \quad \forall t>s>0. 
		\end{equation} 
		It is easy to deduce from the definition of $w_{T}^\prime(.,.)$ and (\ref{FnsFnt}) that 
		\begin{equation*}  
		{w}_{T}^\prime(F_n,\delta) \le {w}_{T}^\prime(\psi,\delta)+\frac{C}{n}.
		\end{equation*} 
		Note that, since $\psi \in D([0, +\infty))$, ${w}_{T}^\prime(\psi,\delta) \rightarrow 0$ as $\delta\rightarrow 0$ (see Sect. 12, p. 123 in \cite{Bill}). Thus, the desired result follows. 
	\end{proof}
	
	We are now ready to state the main result of this section.
	\begin{proposition}
		As $n\longrightarrow +\infty$, $\{F_n(t), t>0\} \longrightarrow \{F(t), t>0\}$ in $D([0, +\infty))$, where $F$ is the unique solution of \eqref{ef1}. 
	\end{proposition}
	\begin{proof}
		From Proposition \ref{Fncompact}, we deduce that at least along a subsequence (but we use the same notation for the subsequence as for the sequence), $F_n$ converges towards a limit denoted by $J$ where $J$ is continuous on the right and admits a limit on the left. In order to show that $F = J$, it suffices to prove that $J$ is a solution of equation \eqref{ef1} and then use Proposition \ref{unicite}. Indeed, let us rewrite \eqref{Fnt} in the form 
		\begin{equation*}
		F_n(t) = \mathbb{E}\Bigg[\mathds{1}_{\eta \leq t}\exp\Bigg\{ \int_{0}^{\frac{\lfloor nt \rfloor-1}{n}} \left(F_n\left(\frac{ \lfloor nt \rfloor -\lceil nu\rceil}{n}\right)-1\right)\lambda(u)du\Bigg\}-\int_{\frac{ \lfloor nt \rfloor-1}{n}}^{t}\lambda(u)du\Bigg].
		\end{equation*}
		Thus, it only remains to show that 
		$$F_n(t) \longrightarrow J(t)= \E\left[ \mathds{1}_{\eta \leq t}\exp\Bigg\{\int_{0}^{t}\Big(J(t-u)-1\Big)\lambda(u) du\Bigg\}\right], \ \mbox{as} \ n\rightarrow +\infty$$ 
		to obtain the desired result. To this end, we note that 
		\begin{align*}
		&\int_{0}^t \left|F_n\left(\frac{ \lfloor nt \rfloor -\lceil nu\rceil}{n}\right)- J(t-u) \right| \lambda(u) du\le \\ 
		&\int_{0}^t \left|F_n\left(\frac{ \lfloor nt \rfloor -\lceil nu\rceil}{n}\right)- F_n(t-u) \right|\lambda(u) du+ \int_{0}^t |F_n(t-u)-J(t-u)| \lambda(u)du\\
		&\le \lambda^* \int_{0}^t \psi(t-u) - \psi\left(t-u-\frac{2}{n} \right) du +  \frac{C}{n}  \lambda^* t+  \lambda^* \int_{0}^t  |F_n(t-u)-J(t-u)| du,
		\end{align*}
		where we have used \eqref{FnsFnt} in the last inequality. Since $\psi$ is left continuous and locally bounded, the first term tends to $0$ as $n\to\infty$,  thanks to Lebesgue's dominated convergence theorem.
		The second term tends clearly to $0$. From the convergence in $D$ for the Skorohod topology,  
		$F_n(t-u)\to J(t-u)$ $du$ a.e.. Moreover Lemma \ref{lem2} allows us to use Lebesgue's dominated convergence theorem again, and the result follows. 
\end{proof}

\paragraph{Acknowledgement} The authors thank Aur\'elien Velleret, whose computations confirmed the results shown 
in section \ref{section54}, , and the Centre National de la Recherche Scientifique, which has supported the visits of Anicet Mougabe and Ibrahima Dram\'e in Marseille within the program DSCA, jointly with the Institut de Math\'ematiques de Marseille.

\end{document}